\documentclass[12 pt]{amsart}
\usepackage{amscd,amssymb,amsmath,amsthm}
\input xy
\xyoption{all}
\newtheorem{Proposition}{Proposition}
\newtheorem{Question}{Question}
\newtheorem{Lemma}{Lemma}
\newtheorem{Theorem}{Theorem}
\newtheorem{Theorem*}{Theorem$^*$}
\newtheorem{Proposition*}{Proposition$^*$}

\newcommand{\proj}{\mathbb{P}}

\newcommand{\rarr}{\rightarrow}
\newcommand{\oh}{{\mathcal{O}}}
\newcommand{\com}{\mathbb{C}}

\newcommand{\bpf}{\noindent {\em Proof.} }
\newcommand{\epf}{\qed \vspace{+10pt}}

\begin{document}
\baselineskip=16pt

\title{ The $\kappa$ ring of
the moduli of curves of compact type: I}
\author{R. Pandharipande}
\date{June 2009}

\begin{abstract}
The subalgebra of the tautological ring
of the moduli of 
curves of compact type
generated by the $\kappa$ classes is
studied in all genera. Relations, constructed 
via the virtual geometry of the moduli 
of stable quotients, are used to
obtain minimal sets of generators.
Bases and Betti numbers of the $\kappa$ rings
are computed.
A universality property relating the
higher genus $\kappa$ rings to the
genus 0 rings is stated and proved in a sequel.
The $\lambda_g$-formula for Hodge integrals
arises as
the simplest consequence.
\end{abstract}

\maketitle
\setcounter{tocdepth}{1}
\tableofcontents

\section{Introduction}
\subsection{Curves of compact type}
Let $C$ be a reduced and connected curve over $\com$
with at worst nodal singularities. 
The associated
{\em dual graph} $\Gamma_C$ has vertices corresponding
to the irreducible components of $C$ 
and edges 
corresponding to the nodes.
The curve $C$ is of {\em compact type} if 
$\Gamma_C$ is a tree.
Alternatively, $C$ is of compact type if the
Picard variety of line bundles of fixed multidegree
on $C$ is compact. 

Standard marked points $p_1,\ldots,p_n$ on $C$ must be
distinct and lie in the nonsingular locus.
The pointed curve $(C,p_1,\ldots,p_n)$ is {\em stable} if the
line bundle
$\omega_C(p_1+\ldots+p_n)$ is ample.
Stability implies the condition $2g-2+n>0$ holds.
Let
$$M_{g,n}^c\subset \overline{M}_{g,n}$$ denote 
the open subset of
 genus $g$, $n$-pointed stable curves of compact type.
The complement
$$\overline{M}_{g,n} \setminus M_{g,n}^c = \delta_0$$
is the irreducible divisor of stable curves with a non-disconnecting
node.

Since every nonsingular curve is of compact type,
the inclusion
$$M_{g,n}\subset M_{g,n}^c\ $$
is obtained.
While the Jacobian map 
$$M_{g,n} \rarr A_g$$
from the moduli of nonsingular curves to the moduli of 
principally polarized Abelian varieties does
not extend to $\overline{M}_{g,n}$, the extension
$$M_{g,n} \subset M_{g,n}^c \rarr A_g$$
is easily defined.

\subsection{$\kappa$ classes}

The $\kappa$ classes in the Chow ring{\footnote{Since
the moduli spaces here are Deligne-Mumford stacks, we
will always take Chow rings with 
$\mathbb{Q}$ coefficients.}} $A^*(\overline{M}_{g,n})$
are defined by the following construction.
Let
$$\epsilon: \overline{M}_{g,n+1} \rarr \overline{M}_{g,n}$$
be the universal curve viewed as the ($n+1$)-pointed space, let
$$\mathbb{L}_{n+1} \rarr \overline{M}_{g,n+1}$$
be the line bundle obtained from the cotangent space
of the last marking, 
and let
$$\psi_{n+1} = c_1(\mathbb{L}_{n+1})\ \in A^1(\overline{M}_{g,n+1})$$
be the Chern class.
The $\kappa$ classes, first defined by Mumford, 
are
$$\kappa_i = \epsilon_*(\psi^{i+1}_{n+1})\  \in A^i(\overline{M}_{g,n}), 
\ \ \ i \geq 0\ .$$
The simplest is $\kappa_0$ which equals $2g-2+n$ times the unit
in $A^0(\overline{M}_{g,n})$.
The convention 
$$\kappa_{-1}= \epsilon_*(\psi_{n+1}^0)=0$$
is often convenient. 

The $\kappa$ classes on $M_{g,n}$ and $M_{g,n}^c$
are defined via restriction from $\overline{M}_{g,n}$.
Define the $\kappa$ rings
\begin{eqnarray*}
\kappa^*(M_{g,n}) & \subset & A^*(M_{g,n}),\\
\kappa^*(M_{g,n}^c)& \subset &  A^*(M_{g,n}^c), \\
\kappa^*(\overline{M}_{g,n}) & \subset & A^*(\overline{M}_{g,n}),
\end{eqnarray*}
to be the $\mathbb{Q}$-subalgebras 
generated by the $\kappa$ classes.
Of course,  the $\kappa$ rings are graded by degree.

Since $\kappa_i$ is a  tautological class{\footnote{A discussion
of tautological classes is presented in Section \ref{tay}.}, 
the $\kappa$ rings
are subalgebras of the corresponding tautological rings.
For unpointed nonsingular curves, the $\kappa$ ring equals
the tautological ring,
$$\kappa^*(M_g)= R^*(M_g)\ .$$
The topic of the paper is the compact type case where
the inclusion
$$\kappa^*(M_{g,n}^c) \subset R^*(M_{g,n}^c)$$
is proper even in degree 1.

\subsection{Results}
We present here several results
 about the rings
$\kappa^*(M_{g,n}^c)$. 
The first two yield a minimal set of generators in the
$n>0$ case.

\begin{Theorem}\label{ooo} $\kappa^*(M_{g,n}^c)$ is generated
over $\mathbb{Q}$ by the
classes $$\kappa_1,\kappa_2,\ldots, \kappa_{g-1+\lfloor\frac{n}{2}
\rfloor}.$$
\end{Theorem}

\begin{Theorem} \label{ttt} 
If $n>0$, 
there are no relations
among $$\kappa_1, \ldots, \kappa_{g-1+
\lfloor \frac{n}{2}\rfloor
} \in \kappa^*(M_{g,n}^c)$$
in degrees $\leq g-1+ \lfloor \frac{n}{2}\rfloor
$.
\end{Theorem}

Since $\kappa^*(M_{g,n}^c) \subset R^*(M_{g,n}^c)$, 
the socle and vanishing results for the
 tautological ring \cite{FP,GV}
imply
\begin{equation}\label{vanni}
\kappa^{2g-3+n}(M_{g,n}^c)=\mathbb{Q}, \ \ \ \ \ \ 
\kappa^{>2g-3+n}(M_{g,n}^c)=0\ .
\end{equation}
By Theorem \ref{ttt}, 
all the interesting relations
among the $\kappa$ classes lie in degrees
from $g+
\lfloor \frac{n}{2}\rfloor
$ to $2g-3+n$.

By Theorem 1, the classes $\kappa_1,\ldots, \kappa_{g-1}$ 
generate $\kappa^*(M_g^c)$. Since 
$M_g^c$ is 
excluded in Theorem \ref{ttt}, the possibility of a 
relation among
the $\kappa$ classes in degree $g-1$ is left open. However, no
lower relations exist.

\begin{Proposition} \label{vvww} There are no relations
among $\kappa_1, \ldots, \kappa_{g-1} \in \kappa^*(M_{g}^c)$
in degrees $\leq g-2$ and at most a single relation in
degree $g-1$.
\end{Proposition}

The structure of $\kappa^*(M_g)$ has been studied 
for many years \cite{Mum}.
Faber \cite{Faber} conjectured the
classes $\kappa_1, \ldots, \kappa_{\lfloor\frac{g}{3}\rfloor}$
form a minimal set of generators for $\kappa^*(M_g)$. 
The result was proven in cohomology by
Morita \cite{Mor}, and a second proof, via admissible
covers and valid in Chow, was given by Ionel \cite{Ion}.
A uniform view of $M_g$, $M_g^c$, and $\overline{M}_g$
was proposed in \cite{FP1}, but very few results in
the latter two cases have been obtained.

\subsection{Relations}
Theorem \ref{ooo} is proven by finding sufficiently many
geometric relations among the 
$\kappa$ classes. The method uses the virtual
geometry of the moduli space of stable quotients
introduced in \cite{MOP} and reviewed in 
Section \ref{sq}. Nonstandard moduli spaces of
pointed curves of compact type are required
for the construction.

Following the notation of \cite{MOP}, let $\overline{M}_{g,n|d}$
be the moduli space of genus $g$
curves of compact type with markings
$$ \{p_1\ldots,p_n\} \ \cup \
\{\widehat{p}_{1},\ldots,\widehat{p}_{d}\}
\in C$$
lying the nonsingular locus
and satisfying the conditions
\begin{enumerate}
\item[(i)]
the points $p_i$ are distinct,
\item[(ii)]
the points $\widehat{p}_{j}$ are distinct from the points $p_i$,
\end{enumerate}
with stability given by the ampleness
of $$\omega_C(\sum_{i=1}^n {p}_{i} +\epsilon \sum_{j=1}^d  \widehat{p}_j)$$
for every strictly positive $\epsilon \in {\mathbb{Q}}$.
The conditions allow the points $\widehat{p}_{j}$ and $\widehat{p}_{j'}$ to coincide.
The moduli space  $\overline{M}_{g,n|d}$ is a nonsingular, irreducible,
Deligne-Mumford stack.{\footnote{In fact, $\overline{M}_{g,n|d}$ 
is a special case of the moduli of pointed curves with weights
studied by \cite{Has,LM}.}}

Denote the open locus of curves of compact type by
$$M^c_{g,n|d}\subset \overline{M}_{g,n|d}\ .$$
Consider the universal curve
$$\pi: U \rarr {M}^c_{g,n|d}.$$
The morphism $\pi$ has sections $\sigma_1, \ldots, \sigma_d$
corresponding to the markings $\widehat{p}_1, \ldots, \widehat{p}_d$.
Let 
$$\sigma \subset U$$
be the divisor obtained from the union of the $d$ sections. 
The two rank $d$ bundles on ${M}_{g,n|d}^c$,
$$\mathbb{A}_d = \pi_*(\oh_\sigma), \ \ \ \
\mathbb{B}_d = \pi_*(\oh_\sigma(\sigma)),$$
play important roles in the geometry.

The new relations studied here 
arise from the vanishing of the
Chern classes of the virtual bundle
$\mathbb{A}_d^*- \mathbb{B}_d$ on
$M_{g,n|d}^c$ after push-forward via the proper
forgetful map
$$\epsilon^c: M_{g,n|d}^c \rarr M_{g,n}^c.$$

\begin{Theorem} \label{rrr}
For all $k > n$, 
$$\epsilon^c_* \left( c_{2g-2+k}(\mathbb{A}_d^*- \mathbb{B}_d) \right) = 0  \  
\in A^*(M^c_{g,n}).$$
\end{Theorem}

The proofs of Theorem \ref{rrr} and richer variants
are given in Section \ref{pr3}.
The $\epsilon^c$ push-forwards are calculated by simple
rules explained in Section \ref{calrul}. In particular, we will
see 
Theorem \ref{rrr} yields relations purely among the $\kappa$ classes
on the moduli space $M_{g,n}^c$.

Theorem \ref{ooo} is proven for $M_{g,n}^c$ in Section \ref{evrel}
 by examining the relations of
Theorem \ref{rrr}. 
The coefficient of $\kappa_i$ for $i\geq g-1+
\lfloor\frac{n}{2}\rfloor$
is shown to be nonzero. The method yields an effective evaluation of
the relations.
Theorem \ref{ttt} and Proposition 1 are proven in Section \ref{indy}
by intersection calculations in the tautological ring.

\subsection{Genus 0}
The strategy of Theorem 3 does {\em not} generate
all the relations in $\kappa^*(M_{g}^c)$. The first
example of failure, occurring in genus 5, is discussed in Section
\ref{uniz}.

Since all genus 0 curves are of compact type,
$$M_{0,n}^c=\overline{M}_{0,n}.$$
For emphasis here, we will use the notation $M_{0,n}^c$.
The following  universality property, motivated by the relations
of Theorem 3, gives considerable weight to the
genus 0 case.

Let $x_1,x_2,x_3, \ldots $ be variables with $x_i$ of degree $i$. Let 
$$f\in \mathbb{Q}[x_1,x_2,x_3,\ldots]$$
be {\em any} graded homogeneous polynomial.

\begin{Theorem}\label{mmmj}
If
$f(\kappa_i) = 0 \in \kappa^*(M_{0,n}^c)$, then
$$f(\kappa_i) = 0 \in \kappa^*(M_{g,n-2g}^c)$$
for all genera $g$ for which $n-2g\geq 0$.
\end{Theorem}

We expect 
variants of Theorem 3 to provide
all relations in $\kappa^*(M_{0,n}^c)$.
A precise statement is
given in Section \ref{bbbn}.
The proof of Theorem \ref{mmmj}, obtained
by stable map techniques, is given in the sequel \cite{kap2}.

\subsection{$\lambda_g$-formula}  
\label{lammm}
The rank $g$ Hodge bundle over the moduli space of curves
$$\mathbb{E} \rarr \overline{M}_{g,n}$$
has fiber $H^0(C,\omega_C)$ over $[C,p_1,\ldots,p_n]$.
Let $$\lambda_k = c_k(\mathbb{E})$$
be the Chern classes.
Since $\lambda_g$ vanishes when restricted to $\delta_0$,
we obtain a well-defined
evaluation
$$\phi: A^*(M_{g,n}^c) \rarr \mathbb{Q}$$
given by integration
$$\phi(\gamma) = \int_{\overline{M}_{g,n}} \overline{\gamma} 
\cdot \lambda_g\ ,$$
where $\overline{\gamma}$ is any lift of $\gamma\in A^*(M_{g,n}^c)$ to
$A^*(\overline{M}_{g,n})$. A discussion of the
above evaluation and the associated Gorenstein conjecture
for the tautological ring can be found in \cite{FP,Pand}.

The evaluation $\phi$ is determined on $R^*(M_{g,n}^c)$
by
the $\lambda_g$-formula for descendent
integrals,
$$\int_{\overline{M}_{g,n}} \psi_1^{a_1}
\cdots \psi_n^{a_n} \lambda_g = \binom{2g-3+n}{a_1,\ldots,a_n} \cdot
\int_{\overline{M}_{g,1}} \psi_1^{2g-2}\lambda_g,$$
discovered in \cite{GeP} and proven in \cite{FP}.
Theorem \ref{mmmj} is much stronger.  The $\lambda_g$-formula 
is a direct {consequence} of Theorem \ref{mmmj}
 in the special case where $f$ has degree
equal to
$$\text{dim}_\com (M_{0,n}^c) = n-3.$$
Conjecture 1 may be view as an extension of
the $\lambda_g$-formula from $\mathbb{Q}$ to cycle
classes.

\subsection{Bases and Betti numbers}
Let $P(d)$ be the set of partitions of $d$, and let
$$P(d,k)\subset P(d)$$ be the set of partitions of $d$ into
at most $k$ parts. 
Let $|P(d,k)|$ be the
cardinality. To a partition{\footnote{The parts of
$\mathbf{p}$ are posititve and satisfy $p_1\geq \ldots \geq p_\ell$.}}
$$\mathbf{p}= (p_1,\ldots,p_\ell) \in P(d,k),$$
we associate a $\kappa$ monomial by
$$\kappa_{\mathbf{p}} = \kappa_{p_1} \cdots \kappa_{p_\ell} \in 
\kappa^d(M_{0,n}^c) \ .$$

\begin{Theorem} \label{htt5}
A $\mathbb{Q}$-basis of $\kappa^d(M_{0,n}^c)$ is given by
$$\{ \kappa_{\mathbf{p}} \ | \ \mathbf{p} \in P(d,n-2-d)\ \} \  .$$
\end{Theorem}

For example, if $d\leq \lfloor \frac{n}{2} \rfloor$-1,
then 
$n-2-d \geq d$ and
$$P(d, n-2-d)= P(d).$$
Hence, Theorem \ref{htt5} agrees with Theorem \ref{ttt}.
The Betti number calculation,
$$\text{dim}_{\mathbb{Q}} \ \kappa^d(M_{0,n}^c) \ = \ |P(d,n-2-d)|\ ,$$
is implied by Theorem \ref{htt5}.
The proof of Theorem \ref{htt5}
is given in Section \ref{bbbn}.

The relations of Theorem 3 and variants
provide an indirect approach for
multiplication in the canonical basis of $\kappa^*(M^c_{0,n})$
determined by Theorem \ref{htt5}.

\begin{Question}
Does there exist a direct calculus for multiplication in the
canonical basis of $\kappa^*(M_{0,n}^c)$ ?
\end{Question}

\subsection{Universality}

The universality of Theorem
\ref{mmmj} expresses the higher genus structures as canonical
{\em ring} quotients,
$$\kappa^*(M_{0,2g+n}^c) \stackrel{\iota_{g,n}}{\rarr} 
\kappa^*(M^c_{g,n}) \rarr 0 \ .$$

\begin{Theorem} If $n> 0$, then $\iota_{g,n}$
is an isomorphism. \label{izz}
\end{Theorem}

The proof is given in Section \ref{gg3} via intersection calculations.
The quotient $\iota_{g,0}$ is not always an isomorphism. 
For example, a nontrivial kernel appears
for $\iota_{5,0}$.

\begin{Question}
What is  the  kernel of $\iota_{g,0}$ ?
\end{Question}

Universality appears to be special to the moduli of
compact type curves.
No similar phenomena have been found for $M_g$ or $\overline{M}_g$.

\subsection{Acknowledgments}
Theorem \ref{rrr} was motivated by the study of stable quotients
developed in
\cite{MOP}. Discussions with A. Marian and D. Oprea 
were very helpful.   
Easy exploration of the
relations of Theorem \ref{rrr} was made possible
by code written by C. Faber. Conversation with C. Faber
played an important role.

The author was partially supported by NSF grant
DMS-0500187
 and the Clay institute.
 The research reported here was undertaken
while the author was visiting  MSRI in Berkeley 
and the 
Instituto
Superior T\'ecnico in Lisbon in the spring of
2009.

\section{Stable quotients}
\label{sq}

\subsection{Stability}
Relations in $\kappa(M_{g,n}^c)$ will be obtained
from the virtual geometry of the moduli space of
stable quotients $\overline{Q}_{g,n}(\proj^1,d)$.
We start by reviewing basic definitions
and results of \cite{MOP}.

Let $C$ be a curve{\footnote{All curves here 
are reduced and connected
with at worst nodal singularities.}}
with distinct markings
$p_1,\ldots, p_n$ in the nonsingular locus $C^{ns}$.
Let $q$ be a quotient of the rank $N$ trivial bundle
 $C$,
\begin{equation*}
\com^N \otimes \oh_C \stackrel{q}{\rarr} Q \rarr 0.
\end{equation*}
If the torsion subsheaf $\tau(Q)\subset Q$ 
has support contained in $$C^{ns}\setminus\{p_1,\ldots,p_n\},$$
 then
$q$ is a {\em quasi-stable quotient}. 
 Quasi-stability of $q$ implies the associated
kernel,
\begin{equation*}
0 \rightarrow S \rightarrow
\com^N \otimes \oh_C \stackrel{q}{\rarr} Q \rarr 0,
\end{equation*}
is a locally free sheaf on $C$. Let $r$ 
denote the rank of $S$.

Let $(C,p_1,\ldots,p_n)$ be a pointed curve equipped
with a quasi-stable quotient $q$.
The data $(C,p_1,\ldots,p_n,q)$ determine 
a {\em stable quotient} if
the $\mathbb{Q}$-line bundle 
\begin{equation}\label{aam}
\omega_C(p_1+\ldots+p_n)
\otimes (\wedge^{r} S^*)^{\otimes \epsilon}
\end{equation}
is ample 
on $C$ for every strictly positive $\epsilon\in \mathbb{Q}$.
Quotient stability implies
$2g-2+n \geq 0$.

Viewed in concrete terms, no amount of positivity of
$S^*$ can stabilize a genus 0 component 
$$\proj^1\stackrel{\sim}{=}P \subset C$$
unless $P$ contains at least 2 nodes or markings.
If $P$ contains exactly 2 nodes or markings,
then $S^*$ {\em must} have positive degree.

A stable quotient $(C,p_1,\ldots,p_n,q)$
yields a rational map from the underlying curve
$C$ to the Grassmannian $\mathbb{G}(r,N)$.
 We will
only require the ${\mathbb{G}}(1,2)=\proj^1$ case
for the proof Theorem \ref{rrr}.

\subsection{Isomorphism}
Let $(C,p_1,\ldots,p_n)$ be a pointed curve.
Two quasi-stable quotients
\begin{equation}\label{fpp22}
\com^N \otimes \oh_C \stackrel{q}{\rarr} Q \rarr 0,\ \ \
\com^N \otimes \oh_C \stackrel{q'}{\rarr} Q' \rarr 0
\end{equation}
on $C$ 
are {\em strongly isomorphic} if
the associated kernels 
$$S,S'\subset \com^N \otimes \oh_C$$
are equal.

An {\em isomorphism} of quasi-stable quotients
 $$\phi:(C,p_1,\ldots,p_n,q)\rarr
(C',p'_1,\ldots,p'_n,q')
$$ is
an isomorphism of curves
$$\phi: C \stackrel{\sim}{\rarr} C'$$
satisfying
\begin{enumerate}
\item[(i)] $\phi(p_i)=p'_i$ for $1\leq i \leq n$,
\item[(ii)] the quotients $q$ and $\phi^*(q')$ 
are strongly isomorphic.
\end{enumerate}
Quasi-stable quotients \eqref{fpp22} on the same
curve $C$
may be isomorphic without being strongly isomorphic.

The following result is proven in \cite{MOP} by
Quot scheme methods from the perspective
of geometry relative to a divisor.

\begin{Theorem} The moduli space of stable quotients 
$\overline{Q}_{g,n}({\mathbb{G}}(r,N),d)$ parameterizing the
data
$$(C,\ p_1,\ldots, p_n,\  0\rarr S \rarr
\com^N\otimes \oh_C \stackrel{q}{\rarr} Q \rarr 0),$$
with {\em rank}$(S)=r$ and {\em deg}$(S)=-d$,
is a separated and proper Deligne-Mumford stack of finite type
over $\com$.
\end{Theorem}

\subsection{Structures}\label{strrr}
Over the moduli space of stable quotients, there is a universal
curve
\begin{equation}\label{ggtt}
\pi: U \rarr \overline{Q}_{g,n}({\mathbb{G}}(r,N),d)
\end{equation}
with $n$ sections and a universal quotient
$$0 \rarr S_U \rarr \com^N \otimes \oh_U \stackrel{q_U}{\rarr} Q_U \rarr 0.$$
The subsheaf $S_U$ is locally 
free on $U$ because of the restrictions imposed on the torsion by the
stability condition.

The moduli space $\overline{Q}_{g,n}({\mathbb{G}}(r,N),d)$ is equipped
with two basic types of maps.
If $2g-2+n >0$, then the stabilization of $(C,p_1,\ldots,p_m)$
determines a map
$$\nu:\overline{Q}_{g,n}({\mathbb{G}}(r,N),d) \rightarrow \overline{M}_{g,n}$$
by forgetting the quotient.
For each marking $p_i$, the quotient is locally free over $p_i$, 
and hence determines
an evaluation map
$$\text{ev}_i: 
\overline{Q}_{g,n}({\mathbb{G}}(r,N),d) \rightarrow {\mathbb{G}}(r,N).$$

The general linear group $\mathbf{GL}_N(\com)$ acts on
$\overline{Q}_{g,n}({\mathbb{G}}(r,N),d)$ via 
the standard
action on $\com^N \otimes \oh_C$. The structures
$\pi$, $q_U$,
$\nu$ and the evaluations maps are all $\mathbf{GL}_N(\com)$-equivariant.

\subsection{Obstruction theory}
The moduli of stable
quotients 
 maps 
to the Artin stack of pointed domain curves
$$\nu^A:
\overline{Q}_{g,n}({\mathbb{G}}(r,N),d) \rightarrow {\mathcal{M}}_{g,n}.$$
The moduli  of stable quotients with fixed underlying
curve 
$$(C,p_1,\ldots,p_n) \in {\mathcal{M}}_{g,n}$$
 is simply
an open set of the Quot scheme. 
The following result of \cite{MOP} is obtained from the
standard deformation theory of the Quot scheme.

\begin{Theorem}\label{htr}
The deformation theory of the Quot scheme 
determines a 2-term obstruction theory on
$\overline{Q}_{g,n}({\mathbb{G}}(r,N),d)$ relative to
$\nu^A$
given by ${{RHom}}(S,Q)$.
\end{Theorem}

An absolute 2-term obstruction theory on
$\overline{Q}_{g,n}({\mathbb{G}}(r,N),d)$ is
obtained from Theorem \ref{htr} and the smoothness
of $\mathcal{M}_{g,n}$, see \cite{BF,GP}. The
 analogue of Theorem \ref{htr} for the Quot scheme of a {\it fixed} nonsingular
 curve was observed in \cite {MO}.

The $\mathbf{GL}_N(\com)$-action lifts to the
obstruction theory,
and the resulting virtual class is
defined in $\mathbf{GL}_N(\com)$-equivariant cycle theory,
$$[\overline{Q}_{g,n}({\mathbb{G}}(r,N),d)]^{vir} 
\in A_*^{\mathbf{GL}_N(\com)}
(\overline{Q}_{g,n}({\mathbb{G}}(r,N),d)).$$

\section{Construction of the relations} \label{pr3}
\subsection{$\com^*$-equivariant geometry}
Let $\com^*$ act on $\com^2$ with weights $[0,1]$ on the
respective basis elements.
Let 
$$\proj^1 = \proj(\com^2),$$
and let $0,\infty \in \proj^1$ be the $\com^*$-fixed
points corresponding the eigenspaces of weight 0 and 1
respectively.

There is an induced $\com^*$-action on
$\overline{Q}_{g,n}(\proj^1,d)$.
Since the virtual dimension of $\overline{Q}_{g,n}(\proj^1,d)$ is
$2g-2+2d+n$,
$$[\overline{Q}_{g,n}(\proj^1,d)]^{vir} 
\in A_{2g-2+2d+n}^{\com^*}
(\overline{Q}_{g,n}(\proj^1,d)),$$
see \cite{MOP}.
The $\com^*$-action lifts canonically{\footnote{The particular
$\com^*$-lift to $S_U$ plays an important role
in the calculation.}} 
 to
the universal curve 
$$\pi: U\rarr \overline{Q}_{g,n}(\proj^1,d).$$
and to the universal subsheaf $S_U$.
The higher direct image
$R^1\pi_*(S_U)$ is a vector bundle of rank
$g+d-1$ 
with top Chern class 
$$\mathsf{e}(R^1\pi_*(S_U)) \in 
A^{g+d-1}_{\com^*}(\overline{Q}_{g,n}(\proj^1,d)).$$

\subsection{Relations}
The relations of Theorem \ref{rrr} will be obtained
by  studying  the   class 
$$\Phi_{g,n,d} = \left( \mathsf{e}(R^1\pi_*(S_U)) \ \cup 
\
\prod_{i=1}^n \text{ev}_i^*([\infty]) \right)\ 
\cap \ [\overline{Q}_{g,n}(\proj^1,d)]^{vir} \ .
$$
on the moduli space of
stable quotients. 
A dimension calculation shows
$$\Phi_{g,n,d} \in 
 A_{g-1+d}^{\com^*}
(\overline{Q}_{g,n}(\proj^1,d))\ .$$

Let $2g-2+n>0$, and consider the proper morphism
$$\nu: \overline{Q}_{g,n}(\proj^1,d) \rarr \overline{M}_{g,n}.$$
Let $[1]$ denote the trivial bundle with $\com^*$-weight $1$, and
let $\mathsf{e}([1])$ be the $\com^*$-equivariant first 
Chern class.
The class
\begin{equation}\label{gret}
\nu_*\left(\Phi_{g,n,d}\ \mathsf{e}([1])^k\right) \in A_{g-1+d-k}(\overline{M}_{g,n})
\end{equation}
certainly vanishes in the non-equivariant limit for $k>0$.

We will calculate the push-forward \eqref{gret}
 via $\com^*$-localization
to find relations. 
Theorem \ref{rrr} will be obtained after restriction to
the moduli space 
$$M_{g,n}^c \subset \overline{M}_{g,n}$$
of curves of compact type.

\subsection{$\com^*$-fixed loci}
Since
$\Phi_{g,n,d}\ \mathsf{e}([1])^k $ is a $\com^*$-equivariant 
class, we may calculate the non-equivariant limit of the 
push-forward \eqref{gret} 
by the virtual localization formula \cite{GP}
as applied
in \cite{MOP}. 
We will be interested in the restriction of 
$\nu_*\left( \Phi_{g,n,d}\ \mathsf{e}([1])^k \right)$ to $M_{g,n}^c$.

The first step
is to determine the $\com^*$-fixed loci of $\overline{Q}_{g,n}(\proj^1,d)$.
The full list of $\com^*$-fixed loci is 
indexed by decorated graphs
described in \cite{MOP}. However, we will see
 most loci do {not}
contribute to the localization calculation of
$$\nu_*\left(\Phi_{g,n,d}\ \mathsf{e}([1])^k\right)|_{M_{g,n}^c}$$
by our
specific choices of $\com^*$-lifts.

The {\em principal} component of the $\com^*$-fixed point locus 
$$\overline{Q}_{g,n}(\proj^1,d)^{\com^*}\subset
\overline{Q}_{g,n}(\proj^1,d)$$ is defined as follows.
Consider
\begin{equation}\label{nrr}
\overline{M}_{g,n|d}\ / \ S_d
\end{equation}
where the symmetric group acts by permutation of
the $d$ nonstandard markings.
Given an element
$$[C, p_1,\ldots,p_n,
\widehat{p}_1,
\ldots,\widehat{p}_d] \in \overline{M}_{g,n|d}\ , $$
there is a canonically associated sequence
\begin{equation}\label{jwq2}
0 \rarr \oh_C(-\sum_{j=1}^d \widehat{p}_j) \rarr \oh_C \rarr Q \rarr 0.
\end{equation}
By including $\oh_C$ as the {\em second} factor of
$\com^2 \otimes \oh_C$, we obtain a stable quotient from
\eqref{jwq2}. The corresponding  $S_d$-invariant morphism
$$\iota: \overline{M}_{g,n|d} \rarr 
\overline{Q}_{g,n}(\proj^1,d)$$
surjects
onto the principal component of $\overline{Q}_{g,n}(\proj^1,d)^{\com^*}$.

Let $F\subset  \overline{Q}_{g,n}(\proj^1,d)^{\com^*}$
be a component of the $\com^*$-fixed locus, and let
$[C,p_1,\ldots,p_n,q]\in F$ be
a generic element of $F$:  
\begin{enumerate}
\item[(i)]
If an irreducible component of $C$ lying over
$0\in \proj^1$ has genus $h>0$, then 
$\mathsf{e}(R^1\pi_*(S_U))$ yields the class $\lambda_h$
by the contribution formulas of \cite{MOP}.
Since 
$$\lambda_h|_{M_{h,*}^c} = 0$$
by \cite{VdG}, such loci $F$ have vanishing contribution
to 
$$\nu_*\left(\Phi_{g,n,d}\ \mathsf{e}([1])^k\right)|_{M_{g,n}^c}\ .$$
\item[(ii)]
If an  irreducible component of $C$ lying over
$0\in \proj^1$ 
is incident to more than a single irreducible
component dominating
$\proj^1$, then $\mathsf{e}(R^1\pi_*(S_U))$ vanishes
on $F$ by the $0$ weight space in $\com^2$ associated to $0\in \proj^1$. 
\item[(iii)]
If  $p_i\in C$ lies over
$0\in \proj^1$, 
then $\text{ev}_i^*([\infty])$ vanishes
on $F$. 
\end{enumerate}
By the  vanishings (i-iii) together with the
stability conditions, we conclude the
principal locus \eqref{nrr} is the
{\em only} $\com^*$-fixed component
of $\overline{Q}_{g,n}(\proj^1,d)$
which contributes to
$\nu_*\left(\Phi_{g,n,d}\ \mathsf{e}([1])^k\right)|_{M_{g,n}^c}$.

\subsection{Proof of Theorem \ref{rrr}}
The contribution of the principal component of
$\overline{Q}_{g,n}(\proj^1,d)$ to the push-forward
$\nu_*\left(\Phi_{g,n,d}\ \mathsf{e}([1])^k\right)|_{M_{g,n}^c}$
is
obtained from the
 localization 
formulas of \cite{MOP} together with an
analysis of $\mathsf{e}(R^1\pi_*(S_U))$.

For
 $[C, p_1,\ldots,p_n,
\widehat{p}_1,
\ldots,\widehat{p}_d] \in \overline{M}_{g,n|d}$, the
long exact sequence associated to \eqref{jwq2} yields
$$0 \rarr \com \otimes \oh_C \rarr \oh_{\widehat{p}_1+\ldots +\widehat{p}_d}
\rarr H^1(C,S) \rarr H^1(C,\oh_C) \rarr 0\ .$$
We conclude
$$\mathsf{e}(R^1\pi_*(S_U)) = \frac{\mathsf{e}(\mathbb{E}^*\otimes[1]) \
\mathsf{e}(\mathbb{A}_d \otimes [1])}{\mathsf{e}([1])}$$
on the principal component. 
The evaluation
$$
 \prod_{i=1}^n \text{ev}_i^*([\infty]) \cdot \mathsf{e}([1])^k
 = \mathsf{e}([-1])^n \cdot
\mathsf{e}([1])^{k}$$
is immediate.

By \cite{MOP},
the full localization contribution of the principal component
is therefore
$$
\frac{\mathsf{e}(\mathbb{E}^*\otimes[1])\ 
\mathsf{e}(\mathbb{A}_d \otimes [1])}{\mathsf{e}([1])}
\mathsf{e}([-1])^n 
\mathsf{e}([1])^{k}
\cdot
\frac{{\mathsf{e}(\mathbb{E}^* \otimes [-1])}}{\mathsf{e}([-1])} 
\frac{1}{\mathsf{e}(\mathbb{B}_d\otimes[-1])}\ .
$$
Using the Mumford relation $c(\mathbb{E})\cdot c(\mathbb{E}^*)=1$, we
conclude, in the non-equivariant limit,
$$\nu_*\left(\Phi_{g,n,d}\ \mathsf{e}([1])^k\right)|_{M_{g,n}^c}
=
(-1)^{3g-3+d+n+k}\
\epsilon^c_*\left( c_{2g-2+k}(\mathbb{A}_d^*-\mathbb{B}_d)\right)\ .$$
Since the non-equivariant limit of
$\nu_*\left(\Phi_{g,n,d}\ \mathsf{e}([1])^k\right)|_{M_{g,n}^c}$
vanishes, the
proof of Theorem \ref{rrr} is complete.
\qed

\subsection{Evaluation rules}
\label{calrul}

\subsubsection{Chern classes}\label{ddd}
Associated to each nonstandard marking $\widehat{p}_j$, there is 
cotangent line bundle
$$\widehat{\mathbb{L}}_j \rarr M_{g,n|d}^c\ .$$
Let $\widehat{\psi}_j=c_1(\widehat{\mathbb{L}}_j)$ be the
first Chern class.

The nonstandard markings are allowed by the stability conditions to be coincident.
The {\em diagonal} $$D_{ij} \subset M_{g,n|d}^c$$ is defined to be
the locus where $\widehat{p}_i=\widehat{p}_j$.
Let
$$S_{ij} =\{ \ \ell \ | \ \ell \neq i,j \ \}\ \cup\ \{ \star \}\ .$$ 
The basic isomorphism
$$D_{ij} \stackrel{\sim}{=} M_{g,n|S_{ij}}^c\ .$$
gives the diagonal geometry a recursive structure compatible 
with the cotangent line classes,
$$\widehat{\psi}_\ell|_{D_{ij}} = \widehat{\psi}_\ell\ ,$$
$$\widehat{\psi}_i|_{D_{ij}} = 
\widehat{\psi}_j|_{D_{ij}} = \widehat{\psi}_\star\ .$$
The intersection of distinct diagonals leads to smaller diagonals
$$D_{ij} \cap D_{jk} = D_{ijk}$$
in the obvious sense.
The self-intersection is determined by
\begin{equation}\label{kw3}
[D_{ij}]^2 = -\widehat{\psi}_\star |_{D_{ij}}\ .
\end{equation}
For convenience, let
$$\Delta_i = D_{1,i}+ D_{2,i} + \ldots + D_{i-1,i}$$
with the convention $\Delta_{1}=0$.

The Chern classes of $\mathbb{A}_d$ and $\mathbb{B}_d$
are easily obtained inductively from the sequences
$$0 \rarr \oh_{\sigma_1+\ldots+\sigma_{d-1}}(-\sigma_d) \rarr
           \oh_{\sigma} \rarr \oh_{\sigma_d} \rarr 0,$$
$$0 \rarr \oh_{\sigma_1+\ldots+\sigma_{d-1}}(\sigma_1+\ldots+\sigma_{d-1})
\rarr 
          \oh_{\sigma}(\sigma) \rarr 
\oh_{\sigma_d}(\sigma) \rarr 0$$
on the universal curve $U$ over $M_{g,n}^c$. We find
\begin{eqnarray}
c(\mathbb{A}_d) &= &\prod_{j=1}^d (1-\Delta_i), \label{frrd}\\
c(\mathbb{B}_d) & =& \prod_{j=1}^d (1-\widehat{\psi}_j+ \Delta_i)\ , \nonumber
\end{eqnarray}
see \cite{MOP} for similar calculations.

\subsubsection{Push-forward}

From the Chern class formulas \eqref{frrd} and the diagonal intersection
rules of Section \ref{ddd}, 
$$\epsilon^c_*( c_{2g-2+k}(\mathbb{A}_d - \mathbb{B}_d)) \in A^*(M_{g,n}^c)$$
is canonically a sum of push-forwards of the type
$$\epsilon^c_*\left(\widehat{\psi}_1^{j_1+1}\cdots 
\widehat{\psi}_s^{j_s+1}\right)\in  A^*(M_{g,n}^c)$$ 
along the forgetful maps
$$\epsilon^c_*: M_{g,n|s}^c \rarr M_{g,n}^c$$
associated to the various diagonals.

\begin{Lemma} \label{nnn}
$\epsilon_*\left(\widehat{\psi}_1^{j_1+1}\cdots 
\widehat{\psi}_s^{j_s+1}\right) = \kappa_{j_1}\cdots
\kappa_{j_d}$
 in $A^*({M}^c_{g,n})$.
\end{Lemma}

\bpf
There are forgetful maps
$$\gamma_j: {M}^c_{g,n|s} \rarr M^c_{g,n|1} = {M}^c_{g,n+1},$$
associated to each nonstandard marking where
the  isomorphism on the right follows from the definition stability.
Taking the fiber product over ${M}^c_{g,n}$ of all the $\gamma_j$
yields a birational morphism
$$\gamma:{M}^c_{g,n|s}\rarr {M}^c_{g,n+1} \times_{{M}^c_{g,n}} {M}^c_{g,n+1}
\times_{{M}^c_{g,n}} \cdots \times_{{M}^c_{g,n}} {M}^c_{g,n+1}\ .$$
The morphism $\gamma$ is a small resolution. The exceptional
loci are at most codimension 2 in ${M}^c_{g,n|s}$. Hence,
$$\mu^*(\psi_j) =\widehat{\psi}_j$$
for each nonstandard marking. We see
$$\mu_*\left(\widehat{\psi}_1^{j_1+1}\cdots 
\widehat{\psi}_s^{j_s+1}\right) =
\psi_1^{j_1+1}\cdots \psi_s^{j_s+1}.$$
The result then follows after push-forward to 
$M_{g,n}^c$
by the definition of the $\kappa$ classes.
\epf

By Lemma 1, the relations of Theorem \ref{rrr} are purely among 
the $\kappa$ classes in $A^*(M_{g,n}^c)$.

\subsubsection{Example} 
The $d=1$ case of Theorem \ref{rrr} immediately yields the relations
$$\forall \ k>n, \ \ \   \kappa_{2g-2+k} = 0\  \in A^*(M_{g,n}^c)$$
implied also by the vanishing results \eqref{vanni}.

More interesting relations occur for $d=2$.
By the  Chern class calculation \eqref{frrd},
$$c( \mathbb{A}^*_2 -\mathbb{B}_2) =
 \frac{1+\Delta_1}{1-\widehat{\psi}_1+\Delta_1} \cdot
\frac{1+\Delta_2}{1-\widehat{\psi}_2+\Delta_2}\ .$$
Using the series expansion
\begin{equation}\label{kcx}
\frac{1+x}{1-y+x} = 1+ \sum_{r\geq 0} y(y-x)^r
\end{equation}
and the diagonal intersection rules,
we obtain
\begin{equation*}
c( \mathbb{A}^*_2 -\mathbb{B}_2) =
 \left(1+ \sum_{r\geq 0} \widehat{\psi}_1^{r+1}\right)\cdot\\
\left(1 + \sum_{r\geq 0} \widehat{\psi}_2(\widehat{\psi}^r_2 -(2^r-1)\widehat{\psi}_2^{r-1} \Delta_2)\right)
\end{equation*}
In genus 3 with $n=0$, the $k=1$ case of Theorem \ref{rrr}
concerns
$$c_5( \mathbb{A}^*_2 -\mathbb{B}_2)  = \sum_{r_1+r_2=5} \widehat{\psi}_1^{r_1} \widehat{\psi}_2^{r_2}
 - \sum_{r=1}^4    (2^r-1)
\widehat{\psi}^4_{\star} \Delta_2 \ .$$
The push-forward is easily evaluated 
\begin{eqnarray*}
\epsilon_*^c( c_5( \mathbb{A}^*_2 -\mathbb{B}_2) ) 
& =& 4\kappa_3 + \kappa_1 \kappa_2 + \kappa_2\kappa_1 + 4 \kappa_3 - (1+3+7+15 ) \kappa_3 \\
& = & -18 \kappa_3 + 2 \kappa_1\kappa_2\ .
\end{eqnarray*}
We obtain the nontrivial
relation
$$-18 \kappa_3 + 2 \kappa_1\kappa_2=0\ \ \in A^*(M_{3}^c)\ .$$

\subsection{Richer relations}
The proof of Theorem \ref{rrr} naturally yields a richer set of
relations among the $\kappa$ classes.
The universal curve 
$$\pi:U \rarr M^c_{g,n|d}$$ carries the  basic
divisor classes
$${s} = c_1(S_U^*), \ \ \ \  \omega= c_1(\omega_\pi)$$
obtained from the universal subsheaf $S_U$ and the $\pi$-relative dualizing
sheaf. 

\begin{Proposition}\label{ggttr}
For all $a_i,b_i\geq 0$ and  $k > n$, 
$$\epsilon_* \left( \prod_{i=1}^m \pi_*(s^{a_i} \omega^{b_i}) \cdot
c_{2g-2+k}(\mathbb{A}_d^*- \mathbb{B}_d) \right) = 0  \  
\in A^*(M^c_{g,n}).$$
\end{Proposition}

The proof of Proposition \ref{ggttr} exactly follows the proof
Theorem \ref{rrr}. We leave the details to the reader.
By the rules of Section \ref{calrul}, the relations of Proposition 
\ref{ggttr} are
also purely among the $\kappa$ classes.

\section{Evaluation of the relations}
\label{evrel}

\subsection{Overview} Our goal here to explicitly evaluate
the relations of Theorem \ref{rrr} as polynomials in the  
$\kappa$ classes. By examining the coefficients, we
will obtain a proof of Theorem \ref{ooo}.
\subsection{Term counts}
Consider the total Chern class
\begin{equation}\label{bbt2}
c( \mathbb{A}^*_d -\mathbb{B}_d) =
 \prod_{i=1}^d  \frac{1+\Delta_i}{1-\widehat{\psi}_i + \Delta_i} \ .
\end{equation}
After substituting 
$$\Delta_i = D_{1,i}+ \ldots + D_{i-1,i},$$
we may expand the right side of \eqref{bbt2} fully. 
The resulting expression is a  formal series in the 
$d+ \binom{d}{2}$ variables{\footnote{The
sign on the diagonal variables is chosen
because of the self-intersection formula \eqref{kw3}.}}
$$\widehat{\psi}_1,\ldots, \widehat{\psi}_d, -D_{12},-D_{13}, 
\ldots,- D_{d-1,d}\ .$$ 
Let $M_r^d$ denote the coefficient in degree $r$,
$$c( \mathbb{A}^*_d -\mathbb{B}_d) =\sum_{r=0}^\infty
M_r^d(\widehat{\psi}_i,- D_{ij}).$$

\begin{Lemma} \label{gcd2} After setting all the variables to 1,
$$\sum_{r=0}^\infty M_r^d(\widehat{\psi}_i=1,-D_{ij}=1) \ t^r
\ = \ \frac{1}{1-dt}.$$
\end{Lemma}

\begin{proof}
After setting the variables to 1 in \eqref{bbt2}, we find
$$c_t( \mathbb{A}^*_d -\mathbb{B}_d) =
 \prod_{i=1}^d  \frac{1-(i-1)t}{1-it},$$
which is a telescoping product.
\end{proof}

Lemma \ref{gcd2} may be viewed  counting the number of
terms in the expansion of \eqref{bbt2},
$$M_r^d(\widehat{\psi}_i=1,-D_{ij}=1)\  = \ d^r\ .$$
 The simple
answer will play a crucial role in the analysis.

\subsection{Connected counts}
A monomial in the diagonal variables
\begin{equation} \label{gqq6}
D_{12},D_{13}, \ldots,D_{d-1,d}
\end{equation}
determines a set partition of $\{1, \ldots, d\}$
by the diagonal associations.
For example, the monomial $3D_{12}^2D_{1,3} D_{56}^3$ determines
the set partition
$$\{1,2,3\} \ \cup \ \{4\}\ \cup \ \{5,6\}$$
in the $d=6$ case.
A monomial in the variables \eqref{gqq6} is
{\em connected} if the corresponding
set partition consists of a single part with $d$ elements.

A monomial in the variables
$$\widehat{\psi}_1,\ldots, \widehat{\psi}_d, -D_{12},-D_{13}, 
\ldots, -D_{d-1,d}\ $$ 
is connected if
the corresponding monomial in the diagonal variables
obtained by setting all $\widehat{\psi}_i=1$
is connected.
Let $C^d_r$ be the summand of $M^d_r(\widehat{\psi}_i=1, -D_{ij}=1)$ 
consisting of the contributions of
only the connected monomials.

\begin{Lemma} \label{llgg}
We have
$$\sum_{d=1}^\infty \sum_{r=0}^\infty
C_r^d\  t^r \frac{z^d}{d!}  =
\log\left( 1+\sum_{d=1}^\infty \sum_{r=0}^\infty
d^r t^r\frac{z^d}{d!}
\right)\ .
$$
\end{Lemma}

\begin{proof}
By a standard application of Wick, the connected
and disconnected counts are related by exponentiation,
$$\exp(\sum_{d=1}^\infty \sum_{r=0}^\infty
C_r^d \ t^r\frac{z^d}{d!}) =
1+ \sum_{d=1}^\infty \sum_{r=0}^\infty
M_r^d(\widehat{\psi}_i=1, -D_{ij}=1) \ t^r\frac{z^d}{d!} \ .$$
The right side is then evaluated by Lemma \ref{gcd2}
\end{proof}

\subsection{$C_r^d$ for $r\leq d$}
We may write the series inside the logarithm in Lemma \ref{llgg}
in the following form,
$$F(t,z) = 1+\sum_{d=1}^\infty \sum_{r=0}^\infty
d^r t^r \frac{z^d}{d!}
=
\exp(tz \frac{d}{dz})\  e^z\ .$$
Expanding the exponential of the differential
operator by order in $t$ yields,
\begin{multline*}
F(t,z) = 
 e^z + tz e^z +  t^2(z^2+z)e^z + \\
t^3(z^3+3z^2+z)e^z + t^4(z^4+6z^3+7z^2+z)e^z + \ldots \ .
\end{multline*}
We have proven the following result.

\begin{Lemma}\label{aaaa}
$F(t,z)= e^z \cdot \sum_{r=0}^\infty t^r p_r(z)$
where
$$p_r(z) = \sum_{s=0}^r c_{r,s} z^{r-s}$$
is a degree $r$ polynomial.
\end{Lemma}

By Lemma \ref{aaaa} and the coefficient
evaluation $c_{r,0}=1$, we see
$$\log(F(t,z)) = z + \log( \frac{1}{1-tz}) + \ldots $$
where the dots stand for terms 
of the form $t^rz^d$ with $r>d$.
We obtain the following result.

\begin{Proposition} \label{vvv}
The only nonvanishing $C_r^d$ for $r\leq d$ are
$C_0^1=1$ and
$$\sum_{r=1}^\infty C_r^r \ t^r \frac{z^r}{r!} = 
-\log(1-tz).$$
\end{Proposition}

\subsection{Evaluation}
Let $g$ and $n$ be fixed.
We are interested in calculating
$$R_{g,n}(t,z) = \sum_{d=1}^\infty \epsilon_*^c\left(c_t
( \mathbb{A}^*_d -\mathbb{B}_d)\right) \frac{z^d}{d!}\ .$$
By the straightforward application of 
the evaluation rules of Section \ref{calrul},
we find
\begin{equation}\label{bbx}
R_{g,n}(t,z) =  \exp\left( \sum_{d=1}^\infty
\sum_{r\geq d}^\infty (-1)^{d-1} C_r^d \kappa_{r-d}\ 
t^r z^{d} \right)\ .
\end{equation}
We rewrite \eqref{bbx} after separating out the
$r=d$ terms using  Proposition \ref{vvv} 
and the evaluation $\kappa_0=2g-2+n$,
\begin{eqnarray*}
R_{g,n}(t,-z) & = & \exp\left(  -\sum_{d=1}^\infty
\sum_{r\geq d}^\infty  C_r^d \kappa_{r-d}\ 
t^r z^{d} \right) \\ \nonumber
& = & (1-tz)^{2g-2+n}
\exp\left(  -\sum_{d=1}^\infty
\sum_{r> d}^\infty  C_r^d \kappa_{r-d}\ 
t^r z^{d} \right) 
\ .
\end{eqnarray*}     
The $t^rz^d$ coefficient of $R_{g,n}$ is
a valid relation in $A^*(M_{g,n}^c)$ if
$$r > 2g-2+n.$$      
The above
formula,  taken together with
Lemma \ref{llgg}, provides a very effective
approach to the relations of 
Theorem \ref{rrr}.

\subsection{Proof of Theorem \ref{ooo}}
The generating series for the
coefficients of the singleton $\kappa_{\ell>0}$ in the
$t^{d+\ell}z^d$ terms of $R_{g,n}(t,-z)$ is
\begin{equation}\label{bwq9}
R^\ell_{g,n}(t,-z)\  =\ 
-(1-tz)^{2g-2+n}
\sum_{d=1}^\infty
  C_{d+\ell}^d \ 
(tz)^{d} t^{\ell}\ . 
\end{equation}

In order to analyze the right side of \eqref{bwq9},
we will use Lemma \ref{aaaa}.
For $\ell \geq 0$, let
\begin{equation}\label{vwrt}
 G_\ell(t,z) = \sum_{d=1}^\infty c_{d+\ell,\ell}\  (tz)^d\ .
\end{equation}
By Lemma \ref{aaaa} and Proposition \ref{vvv},
\begin{eqnarray*}\sum_{\ell\geq 0}
\sum_{d=1}^\infty
  C_{d+\ell}^d \ 
(tz)^{d} t^{\ell}& = & \log \left(
\sum_{\ell\geq 0} G_\ell(t,z)\ t^\ell\right) \\ 
& = & \log \left( \frac{1}{1-tz} +
\sum_{\ell\geq 1} G_\ell(t,z)\ t^\ell\right) \\
& = & \log\left(\frac{1}{1-tz}\right) + \log\left(1 +
(1-zt)\sum_{\ell\geq 1} G_\ell(t,z)\ t^\ell\right)
\ .\end{eqnarray*}
So for $\ell>0$,
\begin{multline}
\label{gxx}
\sum_{d=1}^\infty
  C_{d+\ell}^d \ 
(tz)^{d} = \ \text{Coeff}_{t^\ell} 
\left( \log\left(1 +
(1-zt)\sum_{\ell\geq 1} G_\ell(t,z)\ t^\ell\right) \right)\ .
\end{multline}

The behavior of the coefficients $c_{r,s}$ 
is 
easily determined by induction on $s$.

\begin{Lemma} For $r\geq s$, \label{bbb} 
$c_{r,s}= f_s(r)$
where $f_s(r)$ is a polynomial of  degree $2s$ with leading term
$$f_s(r) = \frac{1}{2^s s!} r^{2s}+ \ldots \ .$$
\end{Lemma}

\noindent For example, $f_0(r)=1$ and
$$f_1(r) =  \frac{1}{2}r^2+\frac{1}{2}r \ . $$
We leave the elementary proof of Lemma \ref{bbb} to the reader

From \eqref{vwrt} and Lemma \ref{bbb}, we conclude for $\ell>0$,
$$G_\ell(t,z) =\frac{1}{2^\ell \ell!}
\frac{(2\ell)!}{(1-tz)^{2\ell+1}}+
\sum_{i=0}^{2\ell} \frac{\tilde{c}_{i,\ell}}{(1-tz)^{i}}$$
for $\tilde{c}_{i,\ell} \in \mathbb{Q}$.
Then by \eqref{gxx},
\begin{equation} \label{vrw}
\sum_{d=1}^\infty
  C_{d+\ell}^d \ 
(tz)^{d} = 
\ \text{Coeff}_{t^\ell} 
\left( \log\left(1 +
\sum_{\ell\geq 1} \frac{(2\ell)!!}{(1-tz)^{2\ell}} t^\ell\right) \right)\ldots ,
\end{equation}
where the dots stand for 
finitely many terms of
the form $(1-tz)^{-j}$ where $j<2\ell$.
By Proposition \ref{kwk2} proven in Section \ref{Nonvani} below,
\begin{equation} \label{kkk3}
\sum_{d=1}^\infty
  C_{d+\ell}^d \ 
(tz)^{d} = 
\frac{\alpha_\ell}{(1-zt)^{2\ell}} + \ldots
\end{equation}
with $\alpha_\ell\neq 0$.

We now return to the coefficients of the singleton
$\kappa_{\ell>0}$ in the
$t^{d+\ell}z^d$ terms of $R_{g,n}(t,-z)$.  
By \eqref{bwq9},
\begin{equation}\label{bwq5}
R^\ell_{g,n}(t,-z)\  =\ 
-\alpha_\ell (1-tz)^{2g-2+n-2\ell}t^\ell+\ldots
\end{equation}
where the dots stand for 
finitely many terms of
the form $(1-tz)^{j}t^\ell$ where $j>2g-2+n -2\ell$.
If
\begin{equation}\label{cond}
2g-2+n-2\ell <0,
\end{equation}
then the coefficient of 
$(tz)^d t^\ell$ in $R^\ell_{g,n}$
will be nonzero for 
for all large $d$.
Once
$$d+\ell > 2g-2+n,$$
the corresponding $\kappa$ relation is
valid by Theorem \ref{rrr}. If \eqref{cond} is satisfied,
$\kappa_\ell$ lies in the subring of
$\kappa^*(M_{g,n}^c)$
generated by $\kappa_1, \ldots, \kappa_{\ell-1}$.
\qed

\subsection{Series analysis} \label{Nonvani}
Define the double factorial by
$$(2\ell)!!= \frac{(2\ell)!}{2^\ell \ell!} = (2\ell-1)\cdot (2\ell-3) \cdots 1 \ $$
and let 
$$\phi(x) = 1 + \sum_{\ell  \geq 1} (2\ell)!!\  x^\ell  = 1+ x + 3x^2+ 15 x^3+ \ldots
\ $$
be the generating series.
Define $\alpha_\ell \in \mathbb{Q}$ for $\ell>0$ by
$$\log(\phi)=
\sum_{\ell\geq 1}
\alpha_\ell x^\ell\ .$$
Series expansion yields the first terms 
$$\log(\phi(x)) = 
x + \frac{5}{2} x^2 + \frac{37}{3} x^3 + \frac{353}{4} x^4 + \ldots \ .$$
To complete the proof of Theorem \ref{rrr}, we must
prove the following result.

\begin{Proposition}
\label{kwk2} $\alpha_\ell \neq 0$ for
 all $\ell>0$.
\end{Proposition}

Let $x=y^2$. Then $\phi(x(y))$ satisfies the differential equation
$$y^2\frac{d}{dy} (y\phi)= \phi-1\ .$$
Equivalently,
$$y^3 \frac{d}{dy} \log(\phi) + y^2 -1 = - \frac{1}{\phi}$$
Changing variables back to $x$, we find
\begin{equation}\label{bpt}
2x^2 \frac{d}{dx} \log(\phi) +x -1 = -\frac{1}{\phi}
\end{equation}
Let $\beta_\ell$ denote the
coefficients of the inverse series,
\begin{eqnarray*}
\phi(x)^{-1} & =& 1+\sum_{\ell\geq 0} \beta_\ell x^\ell \\
 & =& 1 -x -2\alpha_1 x^2 - 4\alpha_2 x^3 - 6 \alpha_3x^4 - \ldots\ ,
\end{eqnarray*}
where the second equality is obtained from \eqref{bpt}.

\begin{Lemma} $\beta_\ell \neq 0$ for all $\ell>0$. \label{jjl}
\end{Lemma}

\begin{proof} Series expansion yields
$$\phi(x)^{-1} = 1-x -2x^2-10x^3 -74 x^4 - \ldots \ .$$
We will establish the following two properties for $\ell>0$
by joint induction: 
\begin{enumerate}
\item[(i)] $\beta_\ell <0$,
\item[(ii)] $|\beta_\ell| \leq (2\ell)!!$\ .
\end{enumerate}
By inspection, the conditions hold in the base case $\ell=1$. 

Let $\ell >1$ and 
assume conditions (i)-(ii) hold for all $\ell'<\ell$.
Since $\phi\cdot \phi^{-1}=1$,
\begin{eqnarray}
(2\ell)!! + \beta_\ell & =& - \sum_{k=1}^{\ell-1} (2k)!! \cdot \beta_{\ell-k} 
\label{ffb}\\ \nonumber
& \leq &\sum_{k=1}^{\ell-1} (2k)!! \cdot (2\ell-2k)!! ,
\end{eqnarray}
where the second line uses (ii).
For 
$\frac{\ell}{2} \leq k \leq \ell-1$,
\begin{eqnarray*}
(2k)!! \cdot (2\ell-2k)!! & = &(2\ell)!!\frac{1} {2\ell-1} \
\frac{3}{2\ell-3}\ 
\cdots \frac{2\ell-2k-1}{2k+1} \\
& \leq & (2\ell)!!\frac{1} {2\ell-1}.
\end{eqnarray*}
By putting the two above inequalities together, we obtain
\begin{equation*}
(2\ell)!! + \beta_\ell \leq (\ell-1) \cdot (2\ell)!!\frac{1} {2\ell-1} <(2\ell)!! \ .
\end{equation*}
Hence, $\beta_\ell <0$.
Since also
$$(2\ell)!!+ \beta_\ell >0$$
by the first equality of \eqref{ffb} and (i),  we see
$|\beta_\ell| < (2\ell)!!$.
\end{proof}

Lemma \ref{jjl} and the relation 
$$-2\ell \alpha_\ell = \beta_{\ell+1}$$
together complete the proof of Proposition \ref{kwk2}.

\section{Independence}
\label{indy}

\subsection{Tautological classes} \label{tay}
The moduli space $M_{g,n}^c$ has an algebraic 
 stratification by
topological type.
The push-forward of the $\kappa$ and $\psi$ classes
from the strata generate the {\em tautological ring}
$$R^*(M_{g,n}^c) \subset A^*(M_{g,n}^c)\ ,$$
see \cite{GP2}. Following the Gorenstein philosophy
explained in \cite{FP1}, we will study
the independence of
$$\kappa_1, \ldots, \kappa_{g-1+\lfloor{\frac{n}{2}}\rfloor}
\in R^*(M_{g,n}^c)$$
through degree $g-1+\lfloor{\frac{n}{2}}\rfloor$
by pairing with strata classes.

\subsection{Case $n=1$}
We first prove Theorem \ref{ttt}
for $M_{g,1}^c$. By stability, $g\geq 1$.
To each partition $\mathbf{p}\in P(d)$, we associate
a $\kappa$ monomial,
$$\kappa_{\mathbf{p}} = \kappa_{p_1} \kappa_{p_2} \cdots \kappa_{p_\ell} \ \in R^*(M_{g,1}^c)\ .$$
Theorem \ref{ttt} is equivalent to the independence of the $|P(g-1)|$
monomials
$$\{\ \kappa_{\mathbf{p}} \ | \ \mathbf{p} \in P(g-1) \ \}$$
in $R^*(M_{g,1}^c)$.

To each partition $\mathbf{p} \in P(g-1)$ of
length $\ell$, we associate a codimension $g-1$
stratum $S_\mathbf{p} \subset M_{g,1}^c$ by the following construction.
Start with a chain of elliptic curves $E_i$ of length $\ell+1$ with the marking on the
first,
\begin{equation} \label{fbn}
 E^*_1 - E_2 - E_3 - \ldots - E_\ell-E_{\ell+1}\ .
\end{equation}
The asterisk indicates the marking.
Since $\ell\leq g-1$, such a chain does not exceed genus $g$.
Next, we add elliptic tails{\footnote{An elliptic
tail is an unmarked
elliptic curve meeting the rest of the curve in exactly 1 point.}} to the first $\ell$ 
elliptic components.
To the curve $E_i$, we add $p_i-1$ elliptic tails.
Let $C$ be the resulting curve.
The total genus of $C$ is
$$\ell+1 + (g-1) - \ell = g\ .$$
The number of nodes of $C$ is
$$\ell + (g-1)-\ell = g-1\ .$$
Hence, $C$ determines a codimension $g-1$ stratum $S_{\mathbf{p}} \subset M_{g,1}^c$.

The moduli in $S_{\mathbf{p}}$ is found mainly on the first $\ell$ components
of the original chain \eqref{fbn}. Each such $E_i$ has $p_i+1$ moduli parameters.
All other components (including $E_{\ell+1}$) are elliptic tails with 1
moduli parameter each.

The $\lambda_g$-evaluation on $R^*(M_{g,1}^c)$ discussed in Section \ref{lammm}
yields a pairing on partitions $\mathbf{p}, \mathbf{q} \in P(g-1)$,
$$\mu_g(\mathbf{p}, \mathbf{q}) = \int_{\overline{M}_{g,1}}
\kappa_{\mathbf{p}} \cdot [S_{\mathbf{q}}] \cdot \lambda_g  \ \in \mathbb{Q}\ .$$

\begin{Lemma}\label{haya} For all $g\geq 1$, the matrix $\mu_g$ is
nonsingular.
\end{Lemma}

\begin{proof}
To evaluate the pairing, we first restrict $\lambda_g$ to $S_{\mathbf{q}}$
by distributing a $\lambda_1$ to each elliptic component.
To pair 
$\kappa_{\mathbf{p}}$  with the class $[S_{\mathbf{q}}]\cdot \lambda_g$, we must
 distribute the factors $\kappa_{p_i}$ to the components $E_j$ 
of $S_{\mathbf{q}}$ in all possible
ways. By the dimension constraints imposed by
 the moduli parameters of the components of $S_{\mathbf{q}}$,
we immediately conclude
$$\mu_g(\mathbf{p}, \mathbf{q})= 0$$
unless $\ell(\mathbf{p}) \geq \ell(\mathbf{q})$. Moreover, if
$\ell(\mathbf{p})= \ell(\mathbf{q})$, the pairing vanishes unless
$\mathbf{p}=\mathbf{q}$.

We have already shown $\mu_g$ to be upper-triangular with 
respect to the length partial ordering on $P(g-1)$.
To establish the nonsingularity of $\mu_g$, we must show the
diagonal entries $\mu_g(\mathbf{p},\mathbf{p})$ do not vanish.
Since $\mu_g(\mathbf{p},\mathbf{p})$ is a product
of factors of the form
$$\int_{\overline{M}_{1,p+1}} \kappa_p \lambda_1 = \frac{1}{24},$$
the required nonvanishing holds.
\end{proof}
By Lemma \ref{haya}, the $\kappa$ monomials of degree $g-1$
are independent. The proof of Theorem \ref{ttt} for $M_{g,1}^c$ is
complete.

\subsection{Case $n=2$}
We now consider Theorem \ref{ttt}
for ${M}_{g,2}^c$. By stability, $g \geq 1 $.
We must prove the independence of the $|P(g)|$
monomials
$$\{\ \kappa_{\mathbf{p}} \ | \ \mathbf{p} \in P(g) \ \}$$
in $R^*(M_{g,2}^c)$.

To each partition $\mathbf{p} \in P(g)$ of length $\ell$, 
we associate a codimension $g-1$
stratum $T_\mathbf{p} \subset M_{g,1}^c$ by the following construction.
Start with a chain of elliptic curves $E_i$ of length $\ell$ with the markings
 on the
first and last,
\begin{equation} \label{fbn3}
 E^*_1 - E_2 - E_3 - \ldots - E^*_\ell \ .
\end{equation}
Since $\ell\leq g$, such a chain does not exceed genus $g$.
Next, we add elliptic tails 
to the  $\ell$ 
elliptic components of \eqref{fbn3}.
To the curve $E_i$, we add $p_i-1$ elliptic tails.
Let $C$ be the resulting curve.
The total genus of $C$ is
$$\ell + g - \ell = g$$
The number of nodes of $C$ is
$$\ell-1 + g-\ell = g-1\ .$$
Hence, $C$ determines a codimension $g-1$ stratum $T_{\mathbf{p}} \subset M_{g,1}^c$.

As before, the $\lambda_g$-evaluation on $R^*(M_{g,2}^c)$ 
yields a pairing on partitions $\mathbf{p}, \mathbf{q} \in P(g)$,
$$\nu_g(\mathbf{p}, \mathbf{q}) = \int_{\overline{M}_{g,2}}
\kappa_{\mathbf{p}} \cdot [T_{\mathbf{q}}] \cdot \lambda_g  \ \in \mathbb{Q}\ .$$

\begin{Lemma}\label{hayav} For all $g\geq 1$, the matrix $\nu_g$ is
nonsingular.
\end{Lemma}

The proof is identical to the proof of Lemma \ref{haya}. We leave
the details to the reader.
The proof of Theorem \ref{ttt} for $M_{g,2}^c$ is complete.

\subsection{Proof of Theorem \ref{ttt}} \label{pttt}
To complete the proof of Theorem \ref{ttt},
we must consider the case $n\geq 3$ and
prove the independence of the 
monomials
$$\{\ \kappa_{\mathbf{p}} \ | \ \mathbf{p} \in P(g-1
+\lfloor \frac{n}{2} \rfloor) \ \}$$
in $R^*(M_{g,n}^c)$.

We will relate the question to the established 
cases with 1 and 2
markings.
Let
$$\widehat{g} = g +  \lfloor\frac{n-1}{2} \rfloor, \ \ \widehat{n}=
n - 
2 \lfloor \frac{n-1}{2} \rfloor\ .$$
If $n$ is odd, then $\widehat{n}=1$. If $n$ is even,
$\widehat{n}=2$. Note
$$\widehat{g}-1+\lfloor \frac{\widehat{n}}{2} \rfloor=
 g-1+\lfloor \frac{n}{2} \rfloor\ .$$

To start, assume $\widehat{n}=1$.
We have constructed strata classes in
$M_{\widehat{g},1}^c$ which 
show the independence of the 
monomials
$$\{\ \kappa_{\mathbf{p}} \ | \ \mathbf{p} \in P(\widehat{g}-1) \ \}$$
in $R^*(M_{\widehat{g},1}^c)$.
For each $\mathbf{q} \in P(\widehat{g}-1)$,
the stratum 
$$S_{\mathbf{q}}\subset M_{\widehat{g},1}^c$$
 consists
of a configuration of $\widehat{g}$
elliptic curves.
We construct a corresponding stratum
$$S'_{\mathbf{q}}\subset M_{g,n}^c$$
by the following method.
Choose any subset{\footnote{The
particular choice of subset is not important.}} 
of $\lfloor \frac{n -1}{2}\rfloor$ 
elliptic components of $S_{\mathbf{q}}$.
For each elliptic component $E$ selected,
replace $E$ with a rational component carrying 2
additional markings.{\footnote{The particular
markings chosen are not important.}} The construction trades 
$\lfloor \frac{n -1}{2}\rfloor$ genus
for 2$\lfloor \frac{n -1}{2}\rfloor$ markings.

Theorem \ref{ttt} is implied by
the nonsingularity of 
the $\lambda_g$-pairing between the $\kappa$ monomials of 
degree  $\widehat{g}-1$ and the strata
classes $[S'_{\mathbb{q}}]$.
The proof of the nonsingularity is identical to
the proof of Lemma \ref{haya}.

The $\widehat{n}=2$ case proceeds by exactly the
same method.
Again, elliptic components of the strata
$$T_{\mathbf{q}}\subset M_{\widehat{g},2}^c$$
are
traded for rational components with 2 additional
markings. Theorem \ref{ttt} is deduced
by nonsingularity of the $\lambda_g$-pairing. \qed

\subsection{Proof of Proposition \ref{vvww}} 
Consider $M_g^c$ for $g\geq 2$.
Let 
$$P^*(g-1) \subset P(g-1) \ \setminus\ \{(1,\ldots, 1) \}$$
be the subset excluding the longest partition.
We will first prove the independence of the
monomials
$$\{\ \kappa_{\mathbf{p}} \ | \ \mathbf{p} \in P^*(g-1) \ \}$$
in $R^*(M_{g}^c)$.
The result shows there can be at most a single
$\kappa$ relation in degree $g-1$.

To each partition $\mathbf{p} \in P^*(g-1)$ of length
$\ell\leq g-2$, we associate a codimension $g-2$
stratum $U_\mathbf{p} \subset M_{g}^c$ by the following construction.
Start with a chain of curves of length $\ell+1$,
\begin{equation*} 
 X - E_2 - E_3 - \ldots - E_\ell-E_{\ell+1}\ ,
\end{equation*}
where $X$ has genus 2 and all the $E_i$ are elliptic curves.
Since $\ell\leq  g-2$, such a chain does not exceed genus $g$.
Next, we add elliptic tails
to the first $\ell$ components. Since
 $p_1$ is the greatest part of $\mathbf{p}$, $p_1\geq 2$.
To the curve $X$, we add $p_1-2$ elliptic tails. 
To the curve $E_i$, we add $p_i-1$ elliptic tails for $2\leq i \leq \ell$.
Let $C$ be the resulting curve.
The total genus of $C$ is
$$2+\ell + (g-1) - \ell-1 = g$$
The number of nodes of $C$ is
$$\ell + (g-1)-\ell-1 = g-2\ .$$
Hence, $C$ determines a codimension $g-2$ stratum $U_{\mathbf{p}} \subset M_{g}^c$.

The $\lambda_g$-evaluation on $R^*(M_{g}^c)$ 
yields a 
 pairing on  $\mathbf{p}, \mathbf{q} \in P^*(g-1)$,
$$\omega_g(\mathbf{p}, \mathbf{q}) = \int_{\overline{M}_{g}}
\kappa_{\mathbf{p}} \cdot [U_{\mathbf{q}}] \cdot \lambda_g  \ \in \mathbb{Q}\ .$$
The argument of Lemma \ref{haya} yields the
following result.

\begin{Lemma}\label{thayav} For all $g\geq 2$, the matrix $\omega_g$ is
nonsingular.
\end{Lemma}

The independence of the $\kappa$ monomials in degrees at most $g-2$
is easier and proven in a similar way. 
To each partition $\mathbf{p} \in P(g-2)$ of length
$\ell$, we associate a codimension $g-1$
stratum $U'_\mathbf{p} \subset M_{g}^c$ by the following construction.
Start with a chain of elliptic curves of length $\ell+2$,
\begin{equation*} 
 E_0-E_1 - E_2 - E_3 - \ldots - E_\ell-E_{\ell+1}\ .
\end{equation*}
Since $\ell\leq  g-2$, such a chain does not exceed genus $g$.
Next, we add elliptic tails
to the components. To $E_i$, for  
 $1\leq  i \leq \ell$,
we add $p_i-1$ elliptic tails.
To $E_0$ and $E_{\ell+1}$,
we add nothing.
Let $C$ be the resulting curve.
The total genus of $C$ is
$$\ell+2 + (g-2) - \ell = g$$
The number of nodes of $C$ is
$$\ell+1 + (g-2)-\ell = g-1\ .$$
Hence, $C$ determines a codimension $g-1$ stratum 
$U'_{\mathbf{p}} \subset M_{g}^c$.

The $\lambda_g$-evaluation on $R^*(M_{g}^c)$ 
yields a 
 pairing on  $\mathbf{p}, \mathbf{q} \in P(g-2)$,
$$\omega'_g(\mathbf{p}, \mathbf{q}) = \int_{\overline{M}_{g}}
\kappa_{\mathbf{p}} \cdot [U'_{\mathbf{q}}] \cdot \lambda_g  \ \in \mathbb{Q}\ .$$
Again, the argument of Lemma \ref{haya} yields the
required result.

\begin{Lemma}\label{thayav2} For all $g\geq 2$, the matrix $\omega'_g$ is
nonsingular.
\end{Lemma}

Together, Lemmas \ref{thayav} and \ref{thayav2} complete the proof
of Proposition \ref{vvww}. \qed

\section{Universality of genus 0}
\label{uniz}

\subsection{Genus 5}\label{aaan}
Do the relations of Theorem \ref{rrr}
generate the entire ideal of relations in $\kappa^*(M_{g}^c)$?
Since Proposition \ref{ggttr} contains
the relations of Theorem \ref{rrr}, we may ask the
same question of the richer system.
The answer to these questions is no. The first example occurs in
$\kappa^6(M_5^c)$.

There are 11 $\kappa$ monomials of degree 6.
By the evaluation rules of Section \ref{calrul}, the $\kappa$ relations
in codimension $6$ generated by  
Proposition \ref{ggttr} are the {\em same} for all the rings
$$\kappa^*(M_{5}^c), \ \kappa^*(M_{4,2}^c),  \ 
\kappa^*(M_{3,4}^c), \ \kappa^*(M_{2,6}^c), \ 
\kappa^*(M_{1,8}^c), \ \kappa^*(M_{0,10}^c)\ .$$
On $M^c_{0,10}$, there are 4 basic types{\footnote{There
are several actual divisors of each type depending
on the marking distribution. We select one of each type.}} of boundary divisors 
determined
by the point splittings
$$8+2,\ \ 7+3,\ \ 6+4,\ \ 5+5 .$$
The pairings of these divisors
with the $\kappa$ monomials $$\kappa_6,\  \kappa_{5}\kappa_1, \ \kappa_4\kappa_2,\
\kappa_3^2$$ on $M_{0,10}^c$
are easily seen to determine a nonsingular $4\times 4$ matrix.
Hence, the number of independent $\kappa$ relations in
$\kappa^6(M_{0,10}^c)$ is at most 7. In fact, Proposition \ref{ggttr}
generates 7 independent relations.

The number of divisor classes in $R^*(M_5^c)$ is 3 given by
$\kappa_1$ and the 2 boundary divisors with genus splittings
$4+1$ and $3+2$.
The Gorenstein conjecture for $M_5^c$ predicts
$R^6(M_5^c)$ to have rank 3. The rank of
$R^6(M_5^c)$ can be proven to be 3 via an application{\footnote{We
thank C. Faber for pointing out the argument.}} of
Getzler's relation \cite{G}. Therefore, there {\em must} be at least
8
relations among the $\kappa$ monomials of degree 6 in 
$M_5^c$. We have proven the method of Proposition
\ref{ggttr} does not yield all the $\kappa$ relations in $R^6(M_5^c)$.

\subsection{Genus 0} \label{bbbn}
In \cite{kap2}, a set of relations obtained
from the virtual geometry of the moduli space of
stable maps is proven to
generate  all the
$\kappa$ relations in the rings $\kappa^*(M_{0,n}^c)$.

\begin{Question}\label{bmm2}
Does Proposition \ref{ggttr}  generate all  the
$\kappa$ relations in the rings $\kappa^*(M_{0,n}^c)$?
\end{Question}

The answer to Question \ref{bmm2} is
affirmative at least for $n\leq 12$. We list below 
the Betti polynomials $B_n(t)$ of $\kappa^*(M^c_{0,n})$ for
low $n$.

\begin{eqnarray*}
B_3 &= &1 \\
B_4 &= &1+t \\
B_5 & = & 1+t+t^2 \\
B_6 & = & 1+t + 2t^2 +t^3 \\
B_7 & = & 1 + t + 2t^2 + 2t^3 + t^4 \\
B_8 & = & 1 + t + 2t^2 + 3t^3 + 3t^4 + t^5\\
B_9 & = & 1 + t + 2t^2 + 3t^3 + 4t^4 +3t^5 + t^6\\
B_{10}& = & 1 + t + 2t^2 + 3t^3 + 5t^4 + 5t^5 + 4t^6+   t^7\\
B_{11}& = & 1 + t + 2t^2 + 3t^3 + 5t^4 + 6t^5 + 7t^6+  4t^7+ t^8\\
B_{12}& = & 1 + t + 2t^2 + 3t^3 + 5t^4 + 7t^5 + 9 t^6+   8t^7+ 5t^8 + t^9\\
\end{eqnarray*}


From the table of Betti numbers, a
formula is easily guessed.
Let $$P(d,k)\subset P(d)$$
be the subset of partitions of $d$ of length at most $k$,
and let $|P(d,k)|$ be the order. We see
$$\text{dim}_{\mathbb{Q}}\ \kappa^d(M_{0,n}^c) = |P(d,n-d-2)|$$
holds in all the above cases.

\setcounter{Theorem}{4}
\begin{Theorem} 
A $\mathbb{Q}$-basis of $\kappa^d(M_{0,n}^c)$ is given by
$$\{ \kappa_{\mathbf{p}} \ | \ \mathbf{p} \in P(d,n-2-d)\ \} \  .$$
\end{Theorem}

\begin{proof}
In order for $P(d,n-d-2)$ to be nonempty, we must have $$d\leq n-3.$$
%
We first prove the independence of the $\kappa$ monomials
associated to 
$P(d, n-d-2)$ by intersection 
with strata classes in
$R^{n-3-d}(M_{0,n}^c)$.
To each partition 
$$\mathbf{p} \in P(d, n-d-2),$$
 we associate a codimension $n-3-d$
stratum $V_\mathbf{p} \subset M_{0,n}^c$ by the following construction.
We write the parts of $\mathbf{p}$ as
$$(p_1, \ldots, p_\ell, p_{\ell+1}, \ldots,p_{n-d-2})$$
where $p_{\ell+\delta} = 0$ for $\delta >0$.
Start with a chain of rational curves of length $n-d-2$,
\begin{equation*} 
 R_1 - R_2 - R_3 - \ldots - R_{n-d-2}\ .
\end{equation*}
Next, we add markings{\footnote{The particular
markings chosen are not important.}} to the components:
\begin{enumerate}
\item[$\bullet$]  $p_1+2$ markings to $R_1$,
\item[$\bullet$] $p_i+1$ markings to $R_i$ for $2\leq i \leq n-d-3$, 
\item[$\bullet$]  $p_{n-d-2}+2$ markings to $R_{n-d-2}$,
\end{enumerate}
Let $C$ be the resulting curve.
The total number of markings of $C$ is
$$2+ d + n-d-2 = n\ .$$ 
The number of nodes of $C$ is $n-3-d$.
Hence, $C$ determines a codimension $n-3-d$ stratum 
$V_{\mathbf{p}} \subset M_{0,n}^c$.

A simple analysis following the strategy of the proof
of Lemma \ref{haya} shows the paring on $P(d,n-d-2)$
given by
\begin{equation*}
(\mathbf{p}, \mathbf{q}) \mapsto \int_{M_{0,n}^c} \kappa_{\mathbf{p}}
\cdot [V]_{\mathbf{q}}
\end{equation*}
is  upper-triangular and nonsingular. We conclude 
the $\kappa$ monomials associated to $P(d,n-d-2)$
are linearly independent.

The strata of $M_{0,n}^c$ are indexed by marked trees.
Given a marked tree $\Gamma$ with $n-2-d$ vertices, the
associated stratum 
$$S_\Gamma \subset M_{0,n}^c$$
parameterizes curves $C$ with marked dual graph $\Gamma$.
In other words, $C$ is a
tree of marked rational components
$$R_1,\ldots, R_{n-2-d}\ .$$
To $S_\Gamma$, we associate a partition $\mathbf{q}(\Gamma)\in P(d,n-d-2)$
by the following construction.
Let $\mathsf{m}(R_i)$ and $\mathsf{n}(R_i)$ denote
the numbers of markings and nodes incident to $R_i$.
Let
$$q_i= \mathsf{m}(R_i)+ \mathsf{n}(R_i) -3.$$
By stability, $q_i\geq 0$. After reordering by size,
$$\mathbf{q}(\Gamma)=(q_1, \ldots, q_{n-d-2}) \in P(d,n-d-2)\ .$$

Let $\mathbf{p} \in P(d)$.
The intersection of $\kappa_{\mathbf{p}}$ with a stratum class $S$ 
is obtained by distributing the factors $\kappa_{p_i}$ 
to the components of $S$.
We conclude
\begin{equation}\label{ywwy}
\int_{M_{0,n}^c} \kappa_{\mathbf{p}} \cdot S_\Gamma =
\int_{M_{0,n}^c} \kappa_{\mathbf{p}} \cdot V_{\mathbf{q(\Gamma)}}
\end{equation}
for all $\mathbf{p} \in P(d)$.

By Poincar\'e duality{\footnote{For $M_{0,n}^c$, singular
cohomology and Chow agree.}}, the dimension of
$\kappa^d(M_{0,n}^c)$ is the rank of the intersection pairing
$$ \kappa^d(M_{0,n}^c) \times A^{n-3-d}(M_{0,n}^c) \rarr \mathbb{Q}.$$
The classes of  strata generate $A^{n-3-d}(M_{0,n}^c)$. Moreover,
only the special strata $V_{\mathbf{q}}$ need by considered
by \eqref{ywwy}. So, 
\begin{equation*}
\text{dim}_{\mathbb{Q}}\ \kappa^d(M_{0,n}^c) 
\leq |P(d,n-d-2)|\ .
\end{equation*}
The independence property together with the above 
dimension estimate yields the basis result.
\end{proof}

\pagebreak
\subsection{Proof of Theorem  \ref{izz}} \label{gg3}
\subsubsection{Bound}
By  Theorem \ref{mmmj} (proven in \cite{kap2}),
we have a surjection
$$\kappa^{d}(M_{0,2g+n}^c) \stackrel{\iota_{g,n}}{\rarr} 
\kappa^d(M_{g,n}^c) \rarr 0.$$
By Theorem 5,
to prove $\iota_{g,n}$ is an isomorphism, we need
only establish  
$$\text{dim}_{\mathbb{Q}} \ \kappa^d(M_{g,n}^c) \geq 
|P(d,2g-2+n-d)|\ $$
for $n>0$. We will obtain the bound by refining the 
argument for Theorem 2.

\subsubsection{Dual graph types}
A dual graph of type $A(g_1, \ldots, g_r)$ with $g_i\geq 1$
is a chain of $r$ vertices of genera $g_1, \ldots, g_r$ with 2 markings 
on the ends. The corresponding curves are of the form:
$$C_{g_1}^* - C_{g_2} - \ldots - C_{g_r}^*\ .$$
If $r=1$, the unique vertex carries both markings.

A dual graph of type $B(g_1, \ldots, g_r|h_1, \ldots, h_{r-1})$
with $g_i,  h_j \geq 1$
is comb of $2r-1$ vertices with  1 marking.
The corresponding curves are of the form:
$$
 \begin{array}{ccccccccc}
 C^*_{g_1}&   -&  C_{g_2} &   - &  \ldots  
&   - &   C_{g_{r-1}} & - & C_{g_r}  \vspace{5pt}\\
\vspace{5pt}
    | &    &   | &   &    &    &   | & &  \\
\vspace{5pt}
C_{h_1}&     &  C_{h_2} &    &  \ldots  &    &   
C_{h_{r-1}}\ . &  &   
\end{array}
$$ 
There are $r-1$ vertices of valence 3 and $r$ vertices
of valence 1. The marking is included in the valence count.

\subsubsection{Case $n=1$}
Let $\mathbf{p}\in P(d)$ be a partition of length $\ell=a+b$ with
parts{\footnote{All parts of $\mathbf{p}$
here are positive.}}
$$(p_1,\ldots,p_a, p'_1,\ldots, p'_b),
$$
where the $p_i$ are odd and the $p'_j$ are even.
We see
$$d+\ell = b \mod 2 \ .$$

If $d+\ell$ is odd, then $b=2r-1$ for $r>0$.
Let $\Gamma_{\mathbf{p}}$
be the dual graph obtained by  the following
construction: 
\begin{eqnarray*}
\Gamma_{\mathbf{p}}& =& 
A\left(\frac{p_1+1}{2}, \ldots, \frac{p_a+1}{2}\right)
\\
& & \  | \\ & &
B\left(\frac{p'_1}{2}, \ldots,\frac{p'_{r-1}}{2},
\frac{p'_r}{2}+1\ | \
\frac{p'_{r+1}}{2}+1 ,\ldots,\frac{p'_{2r-1}}{2}+1\right)
\
, 
\end{eqnarray*}
where the graphs are attached at the first marking of $A$
and the unique marking of $B$.
The graph $\Gamma_{\mathbf{p}}$ has a unique marking (obtained
from the second marking of $A$).
The genus of $\Gamma_{\mathbf{p}}$ is easily calculated,
\begin{equation}\label{vvv23}
2g(\Gamma_{\mathbf{p}})-1 = d + a+2r-1 = d+ \ell \ . 
\end{equation}
 If $a=0$, then 
$\Gamma_{\mathbf{p}}$ consists just of $B$, but the
genus and marking results are the same.

The dual graph $\Gamma_{\mathbf{p}}$ determines a stratum
in $M_{g(\Gamma_{\mathbf{p}}),1}^c$ which is 
a product of the moduli spaces,
$$\prod_{v\in \text{Vert}(\Gamma_{\mathbf{p}})} 
M^c_{g(v),\text{val}(v)}  \ \rarr\  M_{g(\Gamma_{\mathbf{p}}),1}^c\ .$$ 
The socle dimensions of $M^c_{g(v),\text{val}{v}}$
for $v\in \text{Vert}(\Gamma_{\mathbf{p}})$
are exactly the parts of $d$.

If $d+\ell$ is even, then $b$ must be even.
If $b>0$, then $$b=2r-1 +1$$ for $r>0$.
Let 
\begin{eqnarray*}
\Gamma_{\mathbf{p}}& =&  
A\left(\frac{p_1+1}{2}, \ldots, \frac{p_a+1}{2}\right) - 
C_{\frac{p'_{2r}}{2}}^* - E \\
& & \  | \\
& & 
B\left(\frac{p'_1}{2}, \ldots,\frac{p'_{r-1}}{2},
\frac{p'_r}{2}+1\ | \
\frac{p'_{r+1}}{2}+1 ,\ldots,\frac{p'_{2r-1}}{2}+1\right)\ .
\end{eqnarray*}
where the graphs $A$ and $B$
are attached at the markings.
The graph $\Gamma_{\mathbf{p}}$ has a unique marking (on
$C^*_{{\frac{p'_{2r}}{2}}}$) and an
elliptic tail $E$.
The genus of $\Gamma_{\mathbf{p}}$ is 
\begin{equation}\label{vvv24}
2g(\Gamma_{\mathbf{p}})-1 = d + a +2r+2-1 = d+ \ell +1 \ . 
\end{equation}
If $a=0$, then  $A$ is empty, but the genus and marking results
are the same.
The socle dimensions of $M^c_{g(v),\text{val}{v}}$
for $v\in \text{Vert}(\Gamma_{\mathbf{p}})$
are exactly the parts of $d$ together with 0 for the elliptic
tail.

If $d+\ell$ is even and $b=0$,
let 
\begin{eqnarray*}
\Gamma_{\mathbf{p}}& =& 
 A\left(\frac{p_1+1}{2}, \ldots, \frac{p_a+1}{2}\right)- E\ . 
\end{eqnarray*}
The graph $\Gamma_{\mathbf{p}}$ has a unique marking (obtained
from the first marking of $A$) and ends in
the elliptic tail $E$.
The genus of $\Gamma_{\mathbf{p}}$ is
\begin{equation}\label{vvv25} 
2g(\Gamma_{\mathbf{p}})-1 = d +a+2-1 = d+ \ell +1 \ . 
\end{equation}
The socle dimensions of $M^c_{g(v),\text{val}{v}}$
for $v\in \text{Vert}(\Gamma_{\mathbf{p}})$
are exactly the parts of $d$ together with 0 for the elliptic
tail.

We now turn to the proof of Theorem \ref{izz} in the $n=1$
case.  We will prove 
\begin{equation}\label{kkp2}
\text{dim}_{\mathbb{Q}} \ \kappa^d(M_{g,1}^c) \geq 
|P(d,2g-1-d)|
\end{equation}
by intersecting $\kappa$ monomials with tautological
classes.

Let
$\mathbf{p} \in P(d, 2g-1 -d)$
be a partition of length $\ell$.
Let $\Gamma_{\mathbf{p}}$ be the dual graph of genus
$g(\Gamma_{\mathbf{p}})$ obtained by the above
constructions. Since 
$$2g-1 \geq d + \ell\ , $$
equations \eqref{vvv23}-\eqref{vvv25}  imply
$$g-g(\Gamma_{\mathbf{p}}) = \delta \geq 0\ .$$

We associate to 
$\mathbf{p}$ a class $w_{\mathbf{p}} \in R^{2g-2-d}(M_{g,1}^c)$
by the following construction.
Let $v^*\in \text{Vert}(\Gamma_{\mathbf{p}})$
be the vertex which carries the marking.
Increase the genus of $v^*$ by $\delta$. The resulting
graph determines a stratum 
$$W_{\mathbf{p}} \subset M_{g,1}^c$$
 of codimension
$2g(\Gamma_{\mathbf{p}})-2-d$.
Let 
$$w_{\mathbf{p}} = \psi_1^{2\delta} \cdot[ W_{\mathbf{p}}]
\in R^{2g-2-d}(M_{g,1}^c)\ .$$
The pairing on $P(d,2g-1-d)$ given by
\begin{equation}\label{mrt}
(\mathbf{p},\mathbf{q}) \mapsto \int_{\overline{M}_{g,1}} 
\kappa_{\mathbf{p}} \cdot w_\mathbf{q} 
\end{equation}
is upper-triangular. The diagonal elements
are nonvanishing because
$$\int_{\overline{M}_{h}} \kappa_{2h-3} \lambda_h =
\frac{2^{2h-1}-1}{2^{2h-1}} \frac{|B_{2h}|}{(2h)!}
   \neq 0\ ,$$
$$
\int_{\overline{M}_{h,1}} \psi_1^{k}\kappa_{2h-2-k} \lambda_h 
= \binom{2h-1}{k}
\int_{\overline{M}_h} \kappa_{2h-3} \lambda_h \
  \neq 0 \ 
$$
by \cite{FP}. Here, $B_{2h}$ is the Bernoulli number.
Hence, the pairing \eqref{mrt} is nonsingular and the bound \eqref{kkp2} 
is
established.

\subsubsection{Case $n=2$}
We will need an additional dual graph type.
A dual graph of type $\widetilde{B}(g_1, \ldots, g_r|h_1, \ldots, h_{r-1})$
with $g_i,  h_j \geq 1$
is comb of $2r-1$ vertices with  3 markings.
The corresponding curves are of the form:
$$
 \begin{array}{ccccccccc}
 C^*_{g_1}&   -&  C_{g_2} &   - &  \ldots  
&   - &   C_{g_{r-1}} & - & C^{**}_{g_r}  \vspace{5pt}\\
\vspace{5pt}
    | &    &   | &   &    &    &   | & &  \\
\vspace{5pt}
C_{h_1}&     &  C_{h_2} &    &  \ldots  &    &   
C_{h_{r-1}}\ . &  &   
\end{array}
$$ 
There are $r$ vertices of valence 3 and $r-1$ vertices
of valence 1. The marking is included in the valence count.

As before,
let $\mathbf{p}\in P(d)$ be a partition of length $\ell=a+b$ with
parts
$$(p_1,\ldots,p_a, p'_1,\ldots, p'_b),
$$
where the $p_i$ are odd and the $p'_j$ are even.

If $d+\ell$ is even, then $b$ must be even.
If $b>0$, then $$b=2r-1 +1$$ for $r>0$.
Let 
\begin{eqnarray*}
\widetilde{\Gamma}_{\mathbf{p}}& =&  
A\left(\frac{p_1+1}{2}, \ldots, \frac{p_a+1}{2}\right) - 
C_{\frac{p'_{2r}}{2}+1} \\
& & \  | \\
& & 
\widetilde{B}\left(\frac{p'_1}{2}, \ldots,\frac{p'_{r-1}}{2},
\frac{p'_r}{2} \ | \
\frac{p'_{r+1}}{2}+1 ,\ldots,\frac{p'_{2r-1}}{2}+1\right)\ .
\end{eqnarray*}
where the graphs $A$ and $\widetilde{B}$
are attached at the initial markings.
The graph $\widetilde{\Gamma}_{\mathbf{p}}$ has a two markings (on
the extremal component of $\widetilde{B}$).
The genus of $\widetilde{\Gamma}_{\mathbf{p}}$ is 
\begin{equation}\label{vvv28}
2g(\widetilde{\Gamma}_{\mathbf{p}}) = d + a +2r = d+ \ell   \ . 
\end{equation}
If $a=0$, then  $A$ is empty, but the genus and marking results
are the same.
The socle dimensions of $M^c_{g(v),\text{val}{v}}$
for $v\in \text{Vert}(\widetilde{\Gamma}_{\mathbf{p}})$
are exactly the parts of $d$.

If $d+\ell$ is even and $b=0$,
let 
\begin{eqnarray*}
\widetilde{\Gamma}_{\mathbf{p}}& =& 
 A\left(\frac{p_1+1}{2}, \ldots, \frac{p_a+1}{2}\right) \ . 
\end{eqnarray*}
The graph $\widetilde{\Gamma}_{\mathbf{p}}$ has two markings.
The genus of $\widetilde{\Gamma}_{\mathbf{p}}$ is
\begin{equation}\label{vvv29} 
2g(\Gamma_{\mathbf{p}}) = d +a  = d+ \ell  \ . 
\end{equation}
The socle dimensions of $M^c_{g(v),\text{val}{v}}$
for $v\in \text{Vert}(\widetilde{\Gamma}_{\mathbf{p}})$
are exactly the parts of $d$.

If $d+\ell$ is odd, then $b=2r-1$ for $r>0$.
Let 
\begin{eqnarray*}
\widetilde{\Gamma}_{\mathbf{p}}& =& 
A\left(\frac{p_1+1}{2}, \ldots, \frac{p_a+1}{2}\right)-E
\\
& & \  | \\ & &
\widetilde{B}\left(\frac{p'_1}{2}, \ldots,\frac{p'_{r-1}}{2},
\frac{p'_r}{2}\ | \
\frac{p'_{r+1}}{2}+1 ,\ldots,\frac{p'_{2r-1}}{2}+1\right)
\
, 
\end{eqnarray*}
where the graphs $A$ and $\widetilde{B}$ are attached at the initial markings.
The graph $\widetilde{\Gamma}_{\mathbf{p}}$ has two
markings  (on the extremal component of $\widetilde{B}$).
The genus of $\widetilde{\Gamma}_{\mathbf{p}}$ is 
\begin{equation}\label{vvv30}
2g(\widetilde{\Gamma}_{\mathbf{p}}) = d + a+2(r-1)+2 = d+ \ell+1 \ . 
\end{equation}
 If $a=0$, then 
$A$ is empty, but the
genus and marking results are the same.
The socle dimensions of $M^c_{g(v),\text{val}{v}}$
for $v\in \text{Vert}(\Gamma_{\mathbf{p}})$
are exactly the parts of $d$ together with 0 for the
elliptic tail.

The proof of Theorem \ref{izz} now follows the $n=1$
case.  
Let
$$\mathbf{p} \in P(d, 2g -d)$$
be a partition of length $\ell$.
Let $\widetilde{\Gamma}_{\mathbf{p}}$ be the dual graph of genus
$g(\widetilde{\Gamma}_{\mathbf{p}})$ obtained by the above
constructions. Since 
$$2g \geq  d + \ell\ , $$
we see
$g-g(\widetilde{\Gamma}_{\mathbf{p}}) = \delta \geq 0$ .

We associate to 
$\mathbf{p}$ a class $\widetilde{w}_{\mathbf{p}} \in R^{2g-1-d}(M_{g,2}^c)$
by the following construction.
Let $v^*\in \text{Vert}(\widetilde{\Gamma}_{\mathbf{p}})$
be the vertex which carries the first marking.
Increase the genus of $v^*$ by $\delta$. The resulting
graph determines a stratum 
$$\widetilde{W}_{\mathbf{p}} \subset M_{g,2}^c$$
 of codimension
$2g(\widetilde{\Gamma}_{\mathbf{p}})-1-d$.
Let 
$$\widetilde{w}_{\mathbf{p}} = \psi_1^{2\delta} 
\cdot[ \widetilde{W}_{\mathbf{p}}]
\in R^{2g-1-d}(M_{g,2}^c)\ .$$
The pairing on $P(d,2g-d)$ given by
\begin{equation*}
(\mathbf{p},\mathbf{q}) \mapsto \int_{\overline{M}_{g,2}} 
\kappa_{\mathbf{p}} \cdot \widetilde{w}_\mathbf{q} 
\end{equation*}
is upper-triangular and nonsingular as before.
Hence,
$$
\text{dim}_{\mathbb{Q}} \ \kappa^d(M_{g,2}^c) \geq 
|P(d,2g-d)|,
$$
which is the required bound.

\subsubsection{Case $n\geq 3$}
The higher pointed cases are easily reduced to the
$1$ or $2$ pointed cases depending upon the 
parity of $n$. The trading of genera for markings
follows the proof of Theorem 2 in Section \ref{pttt}.
We leave the details to the reader. \qed

\vspace{+8 pt}
\noindent
Department of Mathematics\\
Princeton University\\
rahulp@math.princeton.edu.

\end{document}